\documentclass[11pt]{amsart}
\usepackage[top=4.4cm,bottom=4.4cm,left=3.4cm,right=3.4cm]{geometry}               
\usepackage{amsmath,amscd,amssymb,amsthm}
\usepackage[english]{babel}
\usepackage{mathtools}
\usepackage{enumerate}
\usepackage{setspace}
\usepackage{mathrsfs,makecell}
\usepackage{color,array}
\usepackage[numbers]{natbib}
\usepackage[hidelinks]{hyperref}
\usepackage{algorithm}
\usepackage[noend]{algpseudocode}
\usepackage{algpseudocode}

\DeclarePairedDelimiter{\ceil}{\lceil}{\rceil}

\DeclareMathOperator{\DeHom}{DeHom}

\DeclareMathOperator{\PGL}{PGL}
\DeclareMathOperator{\GL}{GL}
\DeclareMathOperator{\Cost}{Cost}

\def\vF{\mathbb{F}}
\def\vN{\mathbb{N}}
\def\vZ{\mathbb{Z}}

\def\vA{\mathbb{A}}
\def\vQ{\mathbb{Q}}

\def\vP{\mathbb{P}}
\def\vA{\mathbb{A}}

\theoremstyle{definition}
\newtheorem{definition}{Definition}[section]
\newtheorem{remark}[definition]{Remark}

\newtheorem{example}[definition]{Example}

\theoremstyle{plain}
\newtheorem{theorem}[definition]{Theorem}
\newtheorem{corollary}[definition]{Corollary}
\newtheorem{lemma}[definition]{Lemma}
\newtheorem{proposition}[definition]{Proposition}

\newtheorem{question}{Question}
\newtheorem*{theorem*}{Theorem}

\makeindex

\author[D. Goldfeld]{Dorian Goldfeld}
\address{Department of Mathematics\\
Columbia University \\
2990 Broadway\\ 
New York NY 10027, USA
}
\email{dg15@columbia.edu}

\author[G. Micheli]{Giacomo Micheli}
\address{Mathematical Institute\\
University of Oxford\\
Woodstock Rd\\ 
Oxford OX2 6GG, United Kingdom
}
\email{giacomo.micheli@maths.ox.ac.uk}

\title[Fractional Jumps]{The algebraic theory of Fractional Jumps}

\thanks{The author was supported by the Swiss National Science Foundation grant number 171249.}
\subjclass[2010]{11T06, 37P25,	11K45 .}

\keywords{}

\begin{document}

\begin{abstract}
In this paper we start by briefly surveying the theory of Fractional Jumps and transitive projective maps. Then, we give an efficient construction of a fractional jump of a projective map and we extend the compound generator construction for the Inversive Congruential Generator to Fractional jump sequences. In addition, we provide new results on the absolute jump index, on projectively primitive polynomials, and on the explicit description of fractional jump generators.
\end{abstract}

\maketitle

\section{Introduction}

Generating sequences of pseudorandom numbers is of great importance in applied areas and especially in cryptography and for Monte Carlo methods (for example to compute integrals over the reals). The task of generating streams of pseudorandom numbers is closely related to the study of dynamical systems over finite fields, which have been of great interest recently \cite{bib:FMS16, bib:FMS17, bib:GPOS14,bib:HBM17,bib:OS10degree, bib:OS10length, bib:ost10, bib:OPS10}. More in general, for an interesting survey on open problems in arithmetic dynamics see \cite{benedetto2018current}.
Constructions of pseudorandom number generators are studied for example in \cite{bib:eich91,bib:eich92,bib:eich93,bib:EHHW98,bib:EMG09,bib:NS02,bib:NS03,bib:TW06,bib:winterhof10}. This paper focuses on one of the most recent ones, provided in \cite{bib:AGM17}. In a nutshell, \cite{bib:AGM17} provides a new construction of pseudorandom number sequences using the theory of transitive projective maps. 
From an applied point of view, the interest of this new construction relies on the fact that it costs asymptotically less to compute than the classical Inversive Congruential Generator sequence \cite[Section 7]{bib:AGM17} and also achieves the same  discrepancy bounds as the ICG (see \cite[Section 6]{bib:AGM17}). 
From a purely mathematical perspective, the theory of Fractional Jumps is intimately connected with different areas of mathematics such as finite projective geometry, field theory, additive and analytic number theory, and can turn it into a very rich area of research.

The main task of this paper is to summarise the theory of the fractional jump (FJ) construction and complete some mathematical aspects which were left open in the previous papers. 
Finally, we also show that the the compound construction for the Inversive Congruential Generator (ICG) nicely extends to FJs. Also, we leave some open questions at the end of the paper.
%

\subsection*{Notation}
Let $q$ be a prime power, $n$ a positive integer, and $\vF_q$ be the finite field of order $q$. Let $\vA^n$ be the affine space over $\vF_q$ (for the purposes of this paper, this can be simply identified with $\vF_q^n$). Let $\vP^n$ be the projective space of dimension $n$ over $\vF_q$.
Fix the standard projective coordinates $X_0, \ldots, X_n$ on $\vP^n$. 
Let $\GL_{n+1}(\vF_q)$ be the group of invertible matrices over $\vF_q$ and $\PGL_{n+1}(\vF_q)$ be the group of projective automorphisms of $\vP^n$.
For the entire paper we fix the canonical decomposition 
\begin{equation*}                       
\vP^n = U \cup H,
\end{equation*}
where
\begin{align*}
U &= \{[X_0: \ldots: X_n] \in \vP^n \, : \, X_n \neq 0\}\cong \vA^n, \\
H &= \{[X_0: \ldots: X_n] \in \vP^n \, : \, X_n = 0\}\cong \vP^{n-1}.
\end{align*}
For a group $G$ and an element $g\in G$ we denote by $o(g)$ the order of $g$. Let $\Psi\in \PGL_{n+1}(\vF_q)$. We can
write $\Psi$ as $[M]$ for some $M=(m_{i,j})_{i,j} \in \GL_{n+1} (\vF_q)$. Let us denote by $\DeHom(\Psi)$ the $n$-tuple of rational functions
\[(f_1,\dots f_n)=\left(\frac{m_{1,n+1}+\sum^n_{j=1}m_{1,j}x_j}{m_{n+1,n+1}+\sum^n_{j=1}m_{n,j}x_j}, \dots, \frac{m_{n,n+1}+\sum^n_{j=1}m_{n,j}x_j}{m_{n+1,n+1}+\sum^n_{j=1}m_{n,j}x_j}\right).\]
When we have an $n$-tuple $f$ of rational functions of degree $1$ with the same denominator $b$, we say that $b$ is the denominator of $f$. Unless otherwise stated all the logarithms are in basis $2$.

\section{The theory of Fractional Jumps}\label{sec:theoryFJs}
In this section we survey the ingredients needed to construct transitive fractional jumps and give new results on projective primitivity.

\subsection{Transitive projective maps}\label{subsec:transitiveproj}
The first ingredient needed is a transitive automorphism of the projective space.
We start by recalling the definition of projectively primitive polynomials, which are closely related to transitive projective automorphisms.
\begin{definition} \label{projectively_primitive_polynomial}
A polynomial $\chi \in \vF_q [x]$ of degree $m$ is said to be \emph{projectively primitive} if the two following conditions are satisfied:
\begin{enumerate}
\item[i)] $\chi$ is irreducible over $\vF_q$,
\item[ii)] for any root $\alpha$ of $\chi$ in $\vF_{q^m} \cong \vF_q[x] / (\chi)$, the class $[\alpha]$ of $\alpha$ in the quotient group $G = \vF_{q^m}^* / \vF_q^*$ generates $G$.
\end{enumerate} 
\end{definition}

\begin{remark}
Clearly, any primitive polynomial is also projectively primitive.
\end{remark}

A characterisation can be derived from \cite[Lemma 2]{bib:AGM18} with $e=1$. 
\begin{proposition} \label{orders_in_quotients}
An irreducible polynomial $\chi\in \vF_q[x]$ of degree $m$ is projectively primitive if and only if $x^{q-1}\in \vF_q[x]/(\chi)$ has order 
$(q^{m}-1)/(q-1)$.
\end{proposition}

In \cite{bib:AGM17} transitive projective maps were characterised, we report the result here for completeness.

\begin{theorem} \label{transitivity_characterisation}\emph{\cite[Theorem 3.4]{bib:AGM17}}
Let $\Psi$ be an automorphism of $\vP^n$ with $\Psi=[M]\in\PGL_{n+1}(\vF_q)$.  
Then, $\Psi$ is transitive on $\vP^n$ if and only if the characteristic polynomial $\chi_M\in \vF_q [x]$ of $M$ is projectively primitive.
\end{theorem}

\begin{remark}\label{rem:finding_transitive}
Theorem \ref{transitivity_characterisation} also implies that to find a transitive projective automorphism of $\vP^n$  one can simply fix $\Psi=[M_f]\in \PGL_{n+1}(\vF_q)$, where $M_f$ is the companion matrix (or any of its conjugates) of a projectively primitive polynomial $f$.
\end{remark}

The following result shows that one can in principle always construct a primitive polynomial from a projectively primitive one. 

\begin{theorem}
A polynomial $f\in \vF_q[x]$ is projectively primitive  if and only if  there exists $\lambda \in \vF_q^*$ such that $f(x/\lambda)$ is primitive.
\end{theorem}
\begin{proof}
If there exists $\lambda\in \vF_q^*$ such that $f(x/\lambda)$ is primitive, then it is obvious that $f$ is projectively primitive. Let us now show the other implication.
Let $\alpha$ be a root of $f$ in its splitting field $\vF_{q^{\deg(f)}}$. We have to find $\lambda$ such that $\lambda \alpha$ has order $q^{\deg(f)}-1$. Recall that for an element $\beta \in  \vF_{q^{\deg(f)}}^*$ we denote by $[\beta]$ its reduction in the quotient group $G=\vF_{q^{\deg(f)}}^*/\vF_q^*$.

First, observe that for any $\lambda\in \vF_q^*$, we have that $N=(q^{\deg(f)}-1)/(q-1)$ divides $o(\lambda\alpha)$ because $N=o([\alpha])=o([\lambda \alpha])$. So if we can find $\lambda\in \vF_q^*$ such that $(\lambda \alpha)^N$ has order $q-1$ we are done.


Choose a multiplicative generator $g$ of $\vF_q^*$ and write $\alpha^N=\mu=g^e$ for some positive integer $e$. Moreover, assume that the choice of $g$ is also such that $e$ is minimal. First, observe that all the prime factors of $e$ divide $q-1$ as otherwise if $p$ is a prime factor of $e$ that does not divide $q-1$, one can rewrite $(g^p)^{e/p}=\mu$, and $g^p$ is again a generator for $\vF_q^*$, contradicting the minimality of $e$.

We now want to prove that $\gcd(N,e)=1$. Suppose the contrary and let $p$ be a prime factor of $\gcd(N,e)$. Consider $\gamma=\alpha^{N/p}$, if we show that $\gamma^{q-1}=1$ we get the contradiction by the definition of $N$ ($N$ is the smallest integer such that $\alpha^N\in \vF_q^*$). But this is obvious:
\[\gamma^{q-1}=\alpha^{N(q-1)/p}=\left(g^{e}\right)^{(q-1)/p}=\left(g^{e/p}\right)^{(q-1)}=1.\]


Since we want that $(\lambda \alpha)^N$ has order $q-1$, we have to select $\lambda$ such that $(\lambda \alpha)^N$ is a multiplicative generator of $\vF_q^*$. Write $\lambda=g^s$ for some $s\in \vN$, then we can write

\[(\lambda \alpha)^N=g^{sN}\alpha^N=g^{sN}\mu=g^{sN+e}.\]

Since $N$ and $e$ are coprime, Dirichlet Theorem on arithmetic progressions applies, therefore we can select $\overline s$ such that $P=\overline{s} N+e$ is a prime larger than $q-1$. The claim follows by observing that if $g$ is a generator for $\vF_q^*$, then $g^P$ is a generator of $\vF_q^*$.

\end{proof}

A direct consequence of the result above is that when $q$ is small, the problems of finding a primitive polynomial or a projectively primitive one are equivalent.

\begin{corollary}\label{cor:projprimconstruction}
Given a monic projectively primitive polynomial $f$ over $\vF_q$, constructing a primitive polynomial costs $O(q\log(q)\log(\deg(f))$ operations in $\vF_q$.
\end{corollary}
\begin{proof}
We first factor $q-1$ as a precomputation, which costs less than $O(\sqrt q)$.
Given a monic projectively primitive polynomial $f$ and one of its roots $\alpha\in \vF_{q^{ \deg(f)}}=\vF_q[x]/(f(x))$, we simply test (for any $\lambda$ in $\vF_q$) if $\beta_\lambda=(\lambda \alpha)^{\frac{q^{\deg(f)}-1}{q-1}}$ has order $q-1$. The cost is then as follows. Observe that the norm of $\alpha$ is given by the degree zero coefficient of $f$, so $\beta=N(\alpha)=\alpha^{\frac{q^{\deg(f)}-1}{q-1}}$ does not have to be computed.
Since $\beta$ lives in $\vF_q$, for any $\lambda\in \vF_q^*$, we check if $\lambda^{\frac{q^{\deg(f)}-1}{q-1}} \beta=\lambda^{\deg(f)}\beta=\beta_\lambda$  has order $q-1$ in $\vF_q^*$. To do that, we simply compute $\beta_\lambda^{(q-1)/r}$, where $r$ runs over all prime divisors of $q-1$, which are at most $O(\log(q))$. The total number of $\vF_q$-operations is then $O(q\log(q)\log(\deg(f)))$, where $O(\log(\deg(f)))$ is the cost of computing $\lambda^{\deg(f)}$.
\end{proof}

We recall now the definition of fractional jump index.

\begin{definition}
Let $\Psi$ be an automorphism of $\vP^n$.
Let $U=\{[X_0,X_1,\dots, X_{n-1},1]: \:\forall i\in \{0,\dots, n-1\} \: X_i\in \vF_q\}\subseteq \vP^n$ and $P\in U$. The \emph{fractional jump index of $\Psi$ at $P$} is
\begin{equation*}
\mathfrak{J}_{P,\Psi} = \min \{k \geq 1 \, : \, \Psi^k (P) \in U\}.
\end{equation*}
The \emph{absolute fractional jump index $\mathfrak{J}$ of $\Psi$} is the quantity
\begin{equation*}
\mathfrak{J}_\Psi = \max \{\mathfrak{J}_{P,\Psi} \, : \, P \in U\}.
\end{equation*}
\end{definition}

In \cite{bib:AGM17} it is shown that for a transitive projective map, the absolute jump index cannot be larger than $n+1$

\begin{proposition}[{\cite[Corollary 4.3]{bib:AGM17}}]
\label{prop:absolute_jump_above}
Let $\Psi\in \PGL_n(\vF_q)$ be transitive. The absolute jump index of $\mathfrak{J}_\Psi$ of $\Psi$ is less than or equal to $n+1$.
\end{proposition}

We can actually prove a stronger result

\begin{theorem}\label{prop:absolute_jump_below}
Let $\Psi\in \PGL_{n+1}(\vF_q)$ be transitive. Then $\mathfrak J_\Psi=n+1$.
\end{theorem}
\begin{proof}
The direction $\mathfrak J_\Psi\leq n+1$ is given by Proposition \ref{prop:absolute_jump_above}. Let us show that $\mathfrak J_\Psi\geq n+1$. Recall  that $H= \{[X_0: \ldots: X_n] \in \vP^n \, : \, X_n = 0\}\cong \vP^{n-1}.$
Let $L$ be the largest integer such that there exists a point $\overline P\in \vP^n$ such that \[\{\Psi(\overline P), \Psi^2(\overline P), \dots \Psi^L(\overline P)\}\subseteq H,\]
so that $\mathfrak J_\psi=L+1$.
Observe that we can always choose $\overline P$ in $U$ because $\Psi$ is transitive: in fact, consider the smallest $\ell$ such that $P'=\Psi^{-\ell} (\overline P)\in U$ (this is possible as $\Psi$ is transitive). Then \[\{\Psi(P'), \Psi^2(P'), \dots \Psi^{L+\ell}(P')\}\subseteq H.\]
This forces $\ell=0$ and therefore $\overline P\in U$.

Set 
\[T=\{P\in \vP^n: \Psi^i(P)\in H \quad \forall i \in \{1,\dots, L\} \}.\]
It is easy to see that $T$ is non-empty by the choice of $L$, and is a projective subspace of $\vP^n$ that intersects $U$, because $\overline P \in U$. We want to show that the dimension of $T$ is zero, so it consists only of one point. 
Consider $\Psi^{L+1}(T)$ (that has the same dimension of $T$) and assume by contradiction that its dimension is greater than or equal to $1$. Then its intersection with $H$ is non-empty as $H$ is a projective hyperplane, so let $Q\in \Psi^{L+1}(T)\cap H$. Set $R=\Psi^{-L-1}(Q)$ and observe that $\Psi^i(R)\in H$ for any $i\in \{1,\dots, L\}$ as $R\in T$, but also $\Psi^{L+1}(R)\in H$ by construction, which is a contradiction by the maximality of $L$.
This forces $\dim \Psi^{L+1}(T)=\dim T=0$ which forces $T=\{\overline P\}$. Now, since $\dim T\geq n-L$ (each of the conditions $\Psi^i(T)\subseteq H$ imposes an equation), this forces $L\geq n$. Therefore $\mathfrak{J}_{\Psi}\geq n+1$.
\end{proof}

\begin{remark}
Transitivity is necessary for the result above to hold: consider for example the non transitive map of $\vP^1$ given by  $[X,Y]\mapsto [X+Y,Y]$. The absolute jump index is $1$ (no point at finite is mapped at infinite).
\end{remark}
%

\subsection{Constructing a Transitive Fractional Jump}\label{subsec:constructingFjs}

The fractional jump of a projective map can be formally defined as follows
\begin{definition}
Let $U=\{[X_0,X_1,\dots, X_{n-1},1]: \:\forall i\in \{0,\dots, n-1\} \: X_i\in \vF_q\}\subseteq \vP^n$ and
\[\pi:\vA^n\longrightarrow U\]
\[(x_1,\dots, x_n)\mapsto [x_1,\dots,x_n,1]. \]
The \emph{fractional jump of $\Psi$} is the map
\[\psi : \vA^n \rightarrow \vA^n\]
\[x \mapsto \pi^{-1} \Psi^{\mathfrak{J}_{\pi(x)}} \pi (x).\]
\end{definition}

\begin{remark}
The fractional jump is clearly well-defined but its definition depends on the point where it is evaluated, which might be an issue if one wants to describe the map globally. Theorem \ref{explicit_form} ensures that this is not the case.
\end{remark}

Obviously, if one starts with a transitive projective automorphism one will get a transitive fractional jump. Interestingly enough,  the converse implication is also true, apart from two degenerate cases, see \cite[Theorem 2]{bib:AGM18} where this issue is settled. We report the result here for completeness

\begin{theorem}\label{thm:fundamental_description}
Let $\Psi$ be an automorphism of $\vP^n$ and let $\psi$ be its fractional jump. Then, $\Psi$ acts transitively on $\vP^n$ if and only if $\psi$ acts transitively on $\vA^n$, unless $q$ is prime and $n = 1$, or $q = 2$ and $n = 2$, with explicit examples in both cases.
\end{theorem}

In \cite{bib:AGM17} an explicit global description of a fractional jump was given.

\begin{theorem}[{\cite[Section 5]{bib:AGM17} or \cite[Theorem 1]{bib:AGM18}}] \label{explicit_form}
Let $\Psi$ be a transitive automorphism of $\vP^n$, and let $\psi$ be its fractional jump. Then, for $i \in \{1, \ldots, n+1\}$ there exist
\[ a_1^{(i)}, \ldots, a_n^{(i)}, b^{(i)} \in \vF_q[x_1, \ldots, x_n] \]
of degree $1$ such that, if
\begin{align*}
U_1 &= \{x \in \vA^n \, : \, b^{(1)} (x) \neq 0\}, \\
U_i &= \{x \in \vA^n \, : \, b^{(i)} (x) \neq 0, \text{ and } b^{(j)} (x) = 0, \,
 \forall j \in \{1, \ldots, i-1\}\}, \\ &\text{for } i \in \{2, \ldots, n+1\}, \\
&\text{and} \\
f^{(i)} &= \bigg( \frac{a_1^{(i)}}{b^{(i)}}, \ldots, \frac{a_n^{(i)}}{b^{(i)}} \bigg),  \\
&\text{for } i \in \{1, \ldots, n+1\},
\end{align*}
then $\psi (x) = f^{(i)} (x)$ if $x \in U_i$. Moreover, the rational maps $f^{(i)}$ can be explicitly computed.
\end{theorem}
\begin{remark}
Observe that the datum of a fractional jump $\psi$ is equivalent to the datum of the vector of degree $1$ polynomials
$(a^{(1)},\dots a^{(n+1)}; b^{(1)},\dots, b^{(n)})$
where $a^{(i)}=(a_1^{(i)},a_2^{(i)},\dots, a_n^{(i)})$.
\end{remark}

\section{Fractional Jumps in Practice}\label{sec:FJinPractice}
In this section we describe some aspects of the practical implementation of fractional jumps.

\subsection{Compact description}\label{subsec:compact}
In this section we give a compact description of a Fractional Jump.
We first need  an ancillary lemma
\begin{lemma}\label{lem:red_welldef}
Let $\Psi=[M]\in \PGL_{n+1}(\vF_q)$ be transitive. For $i\in \{1,\dots, n+1\}$, 
set  $f^{(i)}=\DeHom(M^i)$ and set $b^{(i)}$ to be the denominator of $f^{(i)}$.
Then $b^{(1)}\neq 0$ and for any $i\in \{2,\dots, n+1\}$ we have that  $b^{(i)}
\not\equiv 0 \mod b^{(1)},\dots, b^{(i-1)}$.
\end{lemma}
\begin{proof}
First observe that since $\Psi$ is transitive we have that: 
\begin{itemize}
\item the characteristic polynomial of $M$ is irreducible and equal to the minimal polynomial  $\mu_M$
\item $b^{(1)}$ is different from 1, as otherwise no point at finite is mapped at infinity and therefore the map cannot be transitive on $\vP^n$.
\end{itemize}
Let $j$ be the smallest integer such that 
$b^{(j)}\equiv 0 \mod (b^{(1)},\dots, b^{(j-1)})$. Of course, we can assume $j\leq n+1$. By degree reasons, there
 exist $\lambda_1, \lambda_2, \dots, \lambda_{j-1}\in \vF_q$ such
 that $\left(\sum^{j-1}_{k=1}\lambda_k b^{(k)}\right)-b^{(j)}=0$. But this implies that the
matrix \[N=\left(\sum^{j-1}_{k=1}\lambda_k M^k\right)-M^j=M\left(\left(\sum^{j-1}_{k=1}\lambda_k M^{k-1}\right)-M^{j-1}\right)\] has the last row identically
 zero, so it is not invertible. But since the characteristic polynomial of $M$ is irreducible, any matrix in $\vF_q[M]\setminus \{0\}$ is invertible. This forces $N=0$. But then the polynomial $g=\left(\sum^{j-1}_{k=1}\lambda_k X^{k-1} \right)-X^{j-1}$ is zero at $M$ and therefore divisible by the minimal polynomial $\mu_M$. But since $j-1\leq n$ and $ \mu_M$ has degree $n+1$, we must have $g=0$, which is a contradiction because $g$ has degree $j-1$.
\end{proof}

We are now ready to provide a compact description of a fractional jump.

\begin{algorithm}
\caption{Fractional Jump Generation Algorithm}\label{algorithm:FJGA}
\begin{flushleft}
\textbf{Input}: a projectively primitive morphism $\Psi=[M]=\in \PGL_{n+1}(\vF_q)$\\ 
\textbf{Output}: the fractional jump of $\Psi$ 
\end{flushleft}
\begin{algorithmic}[1]
\State $M^{(1)}\gets M$ \Comment{ $m^{(1)}_{h,k}$ is the $h$-th row, $k$-th column entry of the matrix  $M^{(1)}$. }
\For{$h \in \{1,\dots n\}$}
\State $a_h^{(1)}\gets m_{h,n+1}^{(1)}+\sum^{n}_{k=1}m^{(1)}_{h,k}x_{k}$
\EndFor
\State $b^{(1)} \gets m_{n+1,n+1}^{(1)}+\sum^{n}_{k=1}m^{(1)}_{n+1,k}x_{k}$

\State $a^{(1)}\gets (a_1^{(1)},\dots,  a_n^{(1)})$
\For{$i\in \{2,\dots n+1\} $}
\State $M^{(i)}\gets M^i$ \Comment{ $m^{(i)}_{h,k}$ is the $h$-th row, $k$-th column entry of the matrix  $M^{(i)}$. }
\State $b^{(i)} \gets m_{n+1,n+1}^{(i)}+\sum^{n}_{k=1}m^{(i)}_{n+1,k}x_{k}\mod b^{(1)},b^{(2)},\dots b^{(i-1)}$
\For{$h \in \{1,\dots n\}$}
\State $a_h^{(i)}\gets m_{h,n+1}^{(i)}+\sum^{n}_{k=1}m^{(i)}_{h,k}x_{k}\mod b^{(1)},b^{(2)},\dots b^{(i-1)}$
\EndFor
\State$a^{(i)}\gets (a_1^{(i)},\dots a_n^{(i)})$
\EndFor
\State \Return $(a^{(1)},a^{(2)}, \dots, a^{(n+1)}),(b^{(1)},b^{(2)}, \dots, b^{(n+1)})$
\end{algorithmic}

\end{algorithm}

\begin{theorem}
Storing a fractional jump requires at most $\ceil{\log(q)}(n+1)^2(n+2)/2$ bits.
\end{theorem}
\begin{proof}
Algorithm \ref{algorithm:FJGA} produces a fractional jump from a transitive projective automorphism. Now observe that the bit size of $(a^{(1)},b^{(1)})$ is the same as the bit size of $M$, which is $(n+1)^2\ceil{\log(q)}$. The bit size of $(a^{(2)},b^{(2)})$ is $(n+1)n\ceil{\log(q)}$ as we were able to use the relation $b^{(1)}=0$. More in general, the bit size of $(a^{(i)},b^{(i)})$ is $(n+1)(n+2-i)\ceil{\log(q)}$ as we can use the relation $b^{(1)}=b^{(2)}=\dots=b^{(i-1)}=0$. The process terminates and it is well defined because of Lemma \ref{lem:red_welldef}.
Adding everything up we get
\[\sum^{n+1}_{i=1} (n+1)(n+2-i)\ceil{\log(q)} =\ceil{\log(q)}(n+1)^2(n+2)/2.\]

\end{proof}

\subsection{Expected cost of evaluation} \label{subsec:eval}
Evaluating a fractional jump is a very easy task, as it involves only one inversion in the base field. In this section we compute the \emph{expected cost} of evaluating a fractional jump, essentially weighting the computational cost with the probability that a random point in $\vF_q^n$ is selected.

\begin{definition}\label{def:projprim}
Let $\psi$ be a map on $\vF_q^n$. We define the \emph{expected cost} of computing $\psi$ on $\vF_q^n$ to be
\[\mathbb E[\psi]=q^{-n}\sum_{x\in \vF_q^n} \Cost(\psi,x), \]
where $\Cost(\psi,x)$ denotes the number of binary operations needed to evaluate $\psi$ at $x$.
\end{definition}

We now compute the expected complexity of evaluating a fractional jump sequence in the large field regime, which is the one for which we have the nice discrepancy bounds in \cite[Section 8]{bib:AGM17}.

\begin{algorithm}
\caption{Fractional Jump Evaluation Algorithm}\label{algorithm:FJEA}
\begin{flushleft}
\textbf{Input}: a fractional jump  $\psi$ and a point $y\in \vF_q^n$.\\ 
\textbf{Output}: $\psi(y)$.
\end{flushleft}
\begin{algorithmic}[1]
\For{$i\in \{1,\dots n+1\} $}
\If{$b^{(i)}(y)\neq 0$}
\State $c\gets b^{(i)}(y)^{-1}$
\State $v\gets a^{(i)}(y)$
\State \Return $cv$
\EndIf
\EndFor
\end{algorithmic}
\end{algorithm}


\begin{theorem}\label{thm:compcostfrac}
Let $q$ be a prime, $\Psi=[M]\in \PGL_{n+1}(\vF_q)$ be a transitive projective automorphism, and $\psi$ be its fractional jump. Suppose that $[M]$ has a 
representative in 
$\GL_{n+1}(\vQ)$ having entries in $\{-1,0,1\}$. Suppose that $q\geq n^3$.
The expected cost of evaluating a fractional jump is $O((n+\log\log(q))\log(q)\log\log(q) \log \log \log (q)+ n^2 \log(q))$.
\end{theorem}
\begin{proof}

We want to estimate the average cost of Algorithm \ref{algorithm:FJEA}.
As usual, set $U^{(1)}=\{x\in \vA^n(\vF_q):  b^{(1)}(x)\neq 0\}$ 
and for $i\in\{2,\dots n+1 \}$ set
\[U^{(i)}=\{x\in \vA^n(\vF_q):  b^{(i)}(x)\neq 0, \text{and } b^{(1)}(x)=b^{(2)}(x)=\dots  =b^{(i-1)}(x)=0\},\] and 
\[\mathbb E[\psi]=q^{-n}\sum^{n+1}_{i=1}\sum_{x\in U^{(i)}} \Cost(\psi,x).\]
For $x \in U^{(1)}$,  by the fact that $M$ has small coefficients, evaluating $a^{(1)}$ and $b^{(1)}$ involves at most $O(n^2)$ sums. Therefore, we have that $\Cost(\psi,x)=O(I(q)+nM(q)+ n^2S(q))$, where $I(q)$ is the cost of an inversion, $M(q)$ is the cost of a multiplication in $\vF_q$ and $S(q)$ the cost of an addition in $\vF_q$. 
For $x\in U^{(i)}$ and $i\geq 2$,  evaluating $a^{(i)}$ and $b^{(i)}$ becomes more expensive, as it might involve also $n-1$ multiplications by elements of $\vF_q$ for each component (the coefficients $m^{(i)}_{h,k}$). The final cost of evaluating at $x\in U^{(i)}$ is then 
$\Cost(\psi,x)=O(I(q)+(n+n^2)M(q)+ n^2S(q))$.
Since there are $q^n-q^{n-1}$ elements in $U^{(1)}$ and $q^{n-1}$ in the union of the rest of the $U^{(i)}$'s we have that 
\[\mathbb E[\psi]=O\left(I(q)+nM(q)+ n^2S(q)+\frac{I(q)+(n+n^2)M(q)+ n^2S(q)}{q}\right).\]

Since $q>n^3$ and $I(q),M(q),S(q)$ are all polynomial time operations in $\log(q)$, we have that $\frac{I(q)+(n+n^2)M(q)+ n^2S(q)}{q}=O(1)$ and then
\[\mathbb E[\psi]=O\left(I(q)+nM(q)+ n^2S(q)\right).\] 
Observe that if one uses Fast Fourier transform for multiplication \cite{bib:SchStr71} and Sch\"onhage Algorithm for inversions \cite[Remark 11.1.99]{mullen2013handbook} we have that \[I(q)=M(q)\log\log(q)\] and \[M(q)=\log(q)\log\log(q) \log \log \log (q).\] Adding two integers modulo $q$ simply costs $O(\log(q))$, from which we get the final claim.
\end{proof}

\begin{example}
Fix for example $p=38685626227668133590597803$ and $f=x^3-x-1\in \vF_p[x]$. One can check with a computer algebra system (for example SAGE \cite{bib:sagemath}) that $(p^3-1)/(p-1)$ is a prime number and that $f$ is an irreducible polynomial. It follows directly from Definition \ref{def:projprim} that $f$ is projectively primitive and therefore the projective map produced by its companion matrix  (see Remark \ref{rem:finding_transitive}) verifies the hypothesis of Theorem \ref{transitivity_characterisation}, and thus it generates a transitive fractional jump verifying the hypothesis of Theorem \ref{thm:compcostfrac}. Computationally it is very easy to produce projectively primitive polynomials, but it would also be interesting to give a systematic way to construct them (such as the one using Artin-Schreier jumps in \cite{bib:AGM18}).

\end{example}

\begin{remark}
In terms of expected complexity (and whenever the coefficients are carefully chosen) fractional jumps behave better than ICGs, as we are about to explain. In fact, let us now compare the result of Theorem \ref{thm:compcostfrac} for $n>1$ with $n=1$ which is essentially the case of the ICG (see \cite[Example 2.4]{bib:AGM17}). Evaluating an ICG having small coefficients costs one inversion $O(I(q))$ whether evaluating a Fractional Jump with small coefficients costs averagely $O(I(q)+nM(q)+n^2S(q))$. Notice now that if $q$ is a large prime and $n$ is relatively small we have that $I(q)+nM(q)+n^2S(q)\sim I(q)$. On the other hand, an ICG only generates one pseudorandom point at each iteration, whether instead the Fractional Jump construction generates $n$-pseudorandom points.
\end{remark}

\subsection{Compound Generator for Fractional Jumps}\label{subsec:compound}

In this subsection we show that the compound generator construction for the Inversive Congruential Generator easily extends to a fractional jump and provide an example.

\begin{theorem}
Let $\ell$ and $n$ be positive integers and  $\{p_1,p_2,\dots, p_\ell\}$ be $\ell$ distinct primes. 
For any $i\in \{1,\dots \ell\}$, let $\Psi_i$ be a transitive projective automorphism of 
$\vP^n(\vF_{p_i})$ and  
$\psi_i:\vA^n(\vF_{p_i})\longrightarrow \vA^n(\vF_{p_i})$ be its fractional jump.

Let $N=p_1\cdots p_\ell$ and  $R=\vZ/N\vZ$. There exists 
a transitive map $\psi$ on $R^n$ such that, for any $i\in \{1,\dots n\}$, its reduction modulo $p_i$ is $\psi_i$.
\end{theorem}

\begin{proof}
Let 
\[v_i=\prod^\ell_{\substack{j=1\\j\neq i}}p_j\]
and $r_i$ be a representative modulo $N$ of the inverse of $v_i$ modulo $p_i$.
Set $u_i=v_i r_i$
and $L_i$ the map which takes as input an element of $\vF_{p_i}^n$ and outputs its canonical representative in $\{0,\dots p_i-1\}^n\subseteq R^n$.
Consider the map
\begin{align*}
\psi:R^n & \longrightarrow R^n\\
x & \mapsto \sum^{\ell}_{i=1}u_i\overline \psi_i(x)
\end{align*}
where \[\overline \psi_i(x)= L_{i}(\psi_i(x\mod p_i)).\]
First observe that $\psi$ is well defined, as it is a sum of well defined maps.
We have now to prove that $\psi$ is a bijection. To see this, notice that we have the following diagram
\[\begin{CD}
R^n @>\psi>> R^n\\
@V\pi_iVV    @V\pi_iVV\\
\vF_{p_i}^n @>\psi_i>>   \vF_{p_i}^n
\end{CD}\]
where $\pi_i$ is the natural reduction of $R^n$ modulo $p_i$.
The diagram is commutative thanks to the choice of $u_i$, which is zero modulo $p_j$ for any $j\neq i$, and modulo $p_i$ is equal to $1$.
We want to prove first that $\psi$ is surjective. Let $z\in R^n$ and consider $z_i=\pi_i(z)$. Since  $\psi_i$ is bijective, there exists $x_i\in \vF_{p_i}^n$ such that $\psi_i(x_i)=z_i$. By the Chinese Remainder Theorem we can find $x\in R^n$ such that $x\equiv x_i \mod p_i$  for all $i\in \{1,\dots,\ell\}$. It is now immediate to see that $\psi(x)=z$. So $\psi$ is surjective and therefore bijective as $R^n$ is a finite set.

We have now to show that $\psi$ is transitive. To see this, we will show that the order of an element $\overline x\in R^n$ is zero modulo $p_i^n$ for any $i\in \{1,\dots, \ell\}$, so the claim will follow as the order of $\psi$ at $\overline x$ is at most $N^n$. Suppose that $d$ is a positive integer such that $\psi^d(\overline x)=\overline x$, then applying $\pi_i$ on both sides and using the commutativity of the diagram we have that \[\pi_i(\psi^d(\overline x))=\psi_i^d(\pi_i(\overline x))=\pi_i(\overline x),\]
from which it follows that $d$ must be divisible by $p_i^n$ as $\psi_i$ is transitive.
\end{proof}

\begin{remark}
Notice that also other lifts $L_i$ to $R^n$ would be suitable for the compound generator, not only the canonical one $\vF_{p_i}^n\rightarrow \{0,1,\dots, p_i-1\}^n\subseteq R^n$.
\end{remark}
\begin{example}
To fix the ideas for our constructions, we produce here a small toy example for $R=\vZ/15\vZ$ and $n=2$.
Let us construct first a transitive projective map over $\vP^2(\vF_5)$. For this, consider the polynomial $x^3+3x+3\in \vF_5[x]$ and its companion matrix
\[M=
\begin{pmatrix}
0 & 0 & 3\\
-1 & 0 & 3\\
0 & -1 & 0
\end{pmatrix}.
\]
To compute the fractional jump of $\Psi_1=[M]\in \PGL_{3}(\vF_5)$ we also need the matrices $M^2$ and $M^3$:
\[
M^2=
\begin{pmatrix}
0 & -3 & 0\\
0  &-3 &-3\\
1 &  0 &-3
\end{pmatrix}
\quad
M^3=
\begin{pmatrix}
3 & 0 & 1\\
3  & 3 & 1\\
0 &  3 & 3
\end{pmatrix}.
\]
The fractional jump of $[M]$ is then

\[
\psi_1(x_1,x_2)=\begin{cases}
(\frac{2}{x_2}, \frac{x_1-3}{x_2}) & \text{if $x_2\neq 0$}\\
(0,\frac{2}{x_1+2}) & \text{if $x_2= 0 $ and $x_1\neq 3$ }\\
(0,0) & \text{if $x=(3,0)$} 
\end{cases}
\]

We now need a projectively primitive polynomial of degree $3$ over $\vF_3$. We select $x^3+2x+1\in \vF_3[x]$. Its companion matrix is 
\[M=
\begin{pmatrix}
0 & 0 & 1\\
-1 & 0 & 2\\
0 & -1 & 0
\end{pmatrix}.
\]
Analogously, one computes the fractional jump of $\Psi_2=[M]\in \PGL_{3}(\vF_3)$ obtaining

\[
\psi_2(x_1,x_2)=\begin{cases}
(-\frac{1}{x_2}, \frac{x_1-2}{x_2}) & \text{if $x_2\neq 0$}\\
(0,-\frac{1}{x_1+1}) & \text{if $x_2= 0 $ and $x_1\neq 2$ }\\
(0,0) & \text{if $x=(2,0)$} 
\end{cases}
\]

The compound generator of $\psi_1$ and $\psi_2$ is then 
\[\psi: R^2 \longrightarrow R^2\]
\[ \psi(x_1,x_2)= 6\cdot L_1(\psi_1(x_1\mod 5, x_2 \mod 5))+10\cdot L_2(\psi_2(x_1\mod 3, x_2 \mod 3))\]
where $L_1$ (resp. $L_2$) is the obvious map lifting $\vF_5$ (resp. $\vF_3$) to $\{0,1,2,3,4\}$ (resp. $\{0,1,2\}$) in $\vZ/15\vZ$. One can check directly that $\psi$ is in fact transitive on $R^2$.
\end{example}

\section{Some ideas to achieve unpredictability from a fractional jump sequence}\label{sec:unpredictability}

Since we already have nice (provable) distributional properties of FJs given by the results in \cite{bib:AGM17} (which make Fractional Jumps suitable for Monte Carlo methods for example), in this section we would like to provide some modifications of the fractional jump construction that could be of use for pseudorandom number generation in settings where unpredictability is a critical property (such as cryptography). 
In this setting we have an opponent observing the stream of pseudorandom numbers and he must not be able to reconstruct the generator, or predict next values of the stream.

\begin{remark}\label{rem:reconstruction} 
We would like to observe that the main issue we encounter when we want to use the basic fractional jump construction for pseudorandom number generation in a cryptographic setting is the following: when the base field $\vF_q$ is large, on most of the points of $\vF_q^n$ we act as $n$ rational functions in $n$ variables of degree $1$ (more precisely in the notation of Theorem \ref{thm:fundamental_description} we act as $f^{(1)}$ on all points of $U_1$, which are $q^n-q^{n-1}$. Therefore, for each pseudorandom number we observed, we get a system of linear equations in the coefficients of the rational functions defining $f^{(1)}$. It is therefore expected that in $(n+1)^2$ points we can reconstruct $f^{(1)}$ by solving a linear system (assuming that all the points in the iteration lie all in $U_1$, which is a reasonable assumption as it has size comparable with $q^n$).
\end{remark}
In what follows we describe some constructions which seem to avoid the issue presented in the remark above.

\subsection{Secret prime $q$}\label{sec:secretprimes}
Here we follow the ideas of \cite{anshel1997zeta}. Choose two large odd primes $p,q$ with the property that $p<q$ and $q=kp+2$ if $p\neq 2$ is odd. The designer keeps $q$ secret, constructs a secret full orbit fractional jump $\psi:\vF_q^n\longrightarrow \vF_q^n$, and chooses a secret starting point $u_0\in\vF_q^n$. Consider now the canonical lift $L:\vF_q^{n}\longrightarrow \{0,1,\dots, q-1\}^n$. The pseudorandom sequence is then produced as $L( \psi^m(u_0) ) \mod p$. To avoid the small biases given by the reduction one can use rejection sampling by skipping elements of the sequence $\psi^m(u_0) \mod q$ that have components that are congruent to $q-2$ or $q-1$ modulo $q$. Of course, $p$ should be chosen relatively small compared with $q$.

\subsection{Forcing jumps}\label{sec:forcing}

Let $\psi:\vF_q\longrightarrow \vF_q$ be a fractional jump, $T$ be a subset of 
$\vF_q^n$ roughly of size $(q^n-1)/2$, $T^c$ be its complement.
Define the map
\[\phi(x)=\begin{cases}
\psi(x) \quad \text{if $x\in T$}\\
\psi(\psi(x)) \quad \text{if $x\in T^c$}
\end{cases}\]
The designer keeps $\psi$,  $T$,  $T^c$, and $\phi$, secret and outputs the sequence $\phi^m(0)$. If one wants to reconstruct the fractional jump $\psi$, according to Remark \ref{rem:reconstruction}, one would need to observe at least $(n+1)^2$ iterations of $\psi$. But in this contruction either $\psi$ or $\psi^2$ is used with probability $1/2$, therefore in order to reconstruct $\psi$ the attacker has $2^{(n+1)^2}$ systems to solve, one of which will lead to the reconstruction of $\psi$.  Notice that with this construction the orbit of $\phi$ starting at any point is bounded from below by $q^n/2$.
%

\section{Further research}

In this section we list some questions arising from the theory of fractional jumps.

Of course, any primitive polynomial is also projectively primitive. Moreover, we saw in Corollary \ref{cor:projprimconstruction} that whenever $q$ is small, finding a primitive polynomial or a projectively primitive polynomials are equivalent problems.
\begin{question} For a fixed degree (e.g. $3$), can one produce algorithms to find projectively primitive polynomials similarily to the one in \cite{bib:chou95}? 
\end{question}

Also, it would be very interesting to see attacks to the constructions in Section \ref{sec:unpredictability}
\begin{question}
Are there (non-trivial) attacks to the constructions in the subsections \ref{sec:secretprimes} and \ref{sec:forcing}?
\end{question}

Finally, we ask to compute the linear complexity of fractional jump sequences, i.e. if $\{v_i\}$ is the sequence in $\vF_q^n$ compute good lower bounds for the minimal $N$ such that there exist $c_1,\dots,c_{N-1}\in \vF_q$ such that for all $i\in \{0,\dots, q^n-1\} $ we have $v_{N+i}=\sum^{N-1}_{j=0}c_jv_{i+j}$.
\begin{question}
What is the linear complexity  of fractional jump sequences produced using the methods described in this paper?
\end{question}

Theorem \ref{thm:compcostfrac} ensures that computing a fractional jump sequence arising from a transitive projective automorphism having a representative matrix with small coefficients has small computational cost (in comparison with the Inversive Congruential Generator for example). Theorem  \ref{transitivity_characterisation} implies that the projective automorphism obtained using the companion matrix of a projectively primitive polynomial is transitive. It is therefore natural to ask the following.
\begin{question}
In which cases one can construct a projectively primitive polynomial with small coefficients?
\end{question}
For example the results in \cite{bib:AGM18} ensure that this is always possible in degree $p$ over the finite field $\vF_p$ using $x^p-x+a$.


\section*{Acknowledgements}
The first author was partially supported by Simons Collaboration Grant 567168. The second author was supported by the Swiss National Science Foundation grant number 171249. The second author would like to thank the department of Mathematics of Columbia University for hosting him in September 2018, as most of these ideas were developed during this stay.

The authors would like to thank Philippe Michel, Alessandro Neri, and Violetta Weger for interesting discussions and suggestions.

\bibliographystyle{plain}
\bibliography{biblio}
\end{document}